\documentclass[12pt]{article}

\usepackage{amsthm}
\usepackage{amsmath}
\usepackage{mathrsfs}
\usepackage{stmaryrd}
\usepackage{amssymb}
\usepackage{diagbox}
\usepackage{tikz-cd}
\usepackage{amsfonts}
\usepackage{mathtools}
\usepackage{authblk}
\usepackage{bm}
\usepackage[utf8]{inputenc}
\usepackage{indentfirst}
\usepackage{graphicx}
\usepackage{enumerate}
\usepackage[top=2.54cm, bottom=2.54cm, left=3.28cm, right=3.28cm]{
geometry}
\usepackage{verbatim}
\usepackage[colorlinks,linkcolor=blue]{hyperref}

\newtheorem{thm}{Theorem}
\theoremstyle{definition}
\newtheorem{defn}{Definition}

\newtheorem{prop}{Proposition}
\newtheorem{conj}{Conjecture}
\newtheorem{cor}{Corollary}

\newtheorem{remark}{Remark}

\makeatletter
\newcommand{\proofpart}[2]{
  \par
  \addvspace{\medskipamount}
  \noindent\textbf{Part #1. #2}
  \par\nobreak\smallskip
  \@afterheading
}
\makeatother

\begin{document}

\title{On a conjecture of Nikiforov concerning the minimal $p$-energy of connected graphs}
\author{\small Quanyu Tang$^{\rm a}$\thanks{Email: tang\_quanyu@163.com}\quad\quad Yinchen Liu$^{\rm b}$\thanks{Email: liuyinch23@mails.tsinghua.edu.cn} \quad\quad Wei Wang$^{\rm a}$\thanks{Email: wang\_weiw@xjtu.edu.cn}
\\
{\footnotesize$^{\rm a}$School of Mathematics and Statistics, Xi'an Jiaotong University, Xi'an 710049, China}
\\
{\footnotesize$^{\rm b}$Institute for Interdisciplinary Information Sciences, Tsinghua University, Beijing 100084, China}
}

\date{}
	\maketitle
\begin{abstract}
For a given simple graph \( G \), the \( p \)-energy of \( G \), denoted by \( \mathcal{E}_p(G) \), is defined as the sum of the \( p \)-th power of the absolute values of the eigenvalues of its adjacency matrix. Let \( S_n \) denote the star graph with one internal node and \( n-1 \) leaves. Nikiforov conjectured that for \( 1 < p < 2 \), the connected graph of order \( n \) with the smallest \( p \)-energy is \( S_n \). Recently, this conjecture was proved for bipartite graphs. In this paper, by employing a Coulson-Jacobs-type formula and certain spectral radius results for connected graphs, we completely resolve this conjecture. Furthermore, we establish that the equality condition in the inequality
\(
\mathcal{E}_p(G) \geq \mathcal{E}_p(S_n)
\)
holds if and only if \( G \) is \( S_n \).
\end{abstract}


\section{Introduction}
 We start with some definitions and notations. Throughout this paper, we only consider \emph{simple} graphs, i.e., undirected graphs without loops or multiple edges. Let $G = (V, E)$ be a graph of order $n$ and size $m$. The \textit{adjacency matrix} of $G$ is an $n \times n$ matrix $A(G) = [a_{ij}]$, where $a_{ij} = 1$ if the vertices $v_i$ and $v_j$ are adjacent and $a_{ij} = 0$, otherwise. The \textit{eigenvalues} of $G$ are the eigenvalues of $A(G)$. Since $A(G)$ is a real symmetric matrix, all the eigenvalues of $A(G)$ are real. Let \( \lambda_1, \dots, \lambda_n \) be the eigenvalues of \( G \), collectively known as the \emph{spectrum} of \( G \). Then the \textit{$p$-energy} of $G$ is defined as the sum of the \( p \)th power of the absolute values of the eigenvalues of its adjacency matrix
$$\mathcal{E}_p(G)=\sum^n_{i=1}|\lambda_i|^p,$$
where $p$ is a positive real number. When \( p = 1 \), \( \mathcal{E}_1(G) \) is simply denoted as \( \mathcal{E}(G) \), known as the \textit{energy} of \( G \). Originating from applications in molecular chemistry, graph energy has attracted considerable attention. We refer the reader to \cite{Gutman2001} for classical bounds on graph energy. The \( p \)-energy is also known as the \textit{\( p \)-Schatten energy} in other works, such as \cite{Arizmendi2023b}. In this paper, we follow the terminology used by Akbari et al. \cite{Akbari2020}.

In Nikiforov \cite{Nikiforov2012}, the Schatten $p$-norms of graphs were studied.
\begin{defn}
Let $A$ be an $n \times n$ complex matrix, and let $p \geq 1$. The \emph{Schatten $p$-norm} $\|A\|_{p}$ is given by
\[
\|A\|_{p} := \left( \sigma_1(A)^p + \cdots + \sigma_n(A)^p \right)^{1/p},
\]
where $\sigma_1(A) \geq \sigma_2(A) \geq \dots \geq \sigma_n(A)$ are the singular values of $A$.
\end{defn}

If $G$ is a graph with adjacency matrix $A$, for short we write $\|G\|_{p}$ for $\|A\|_{p}$. Clearly, $\|G\|_{1} = \mathcal{E}(G)$, and the $p$-energy $\mathcal{E}_p(G)$ of a graph $G$ is exactly $\|G\|_{p}^p$.

Let \( P_n \) denote the path graph on \( n \) vertices, and let \( S_n \) denote the star graph with one internal node and \( n-1 \) leaves. It is known that the path graph has the maximal energy among all trees of a given order, while the star graph has the minimal energy among all connected graphs of a given order.

For the \( p \)-energy, there is a distinct difference in behavior between \( 1 < p < 2 \) and \( p > 2 \). This was first noted by Li, Shi, and Gutman \cite{Li2012} as a private communication from Wagner and was later revisited by Nikiforov \cite{Nikiforov2016}, where it was formulated in terms of Schatten \( p \)-norms. Here, we reinterpret it in the language of \( p \)-energy.

\begin{conj}[{\cite{Nikiforov2016}}]\label{conj1}
Let \( G\) be a connected graph of order \(n\).
\begin{itemize}
    \item[(a)] If \( p > 2 \), then \( \mathcal{E}_p(G) \geq \mathcal{E}_p(P_n) \).
    \item[(b)] If \( 1 < p < 2 \), then \( \mathcal{E}_p(G) \geq \mathcal{E}_p(S_n)\).
\end{itemize}
\end{conj}

In \cite{Csikvari2010}, Csikv\'{a}ri proved that among all connected graphs of a given order, the path graph minimizes the number of closed walks of any given length. Nikiforov \cite{Nikiforov2016} later reformulated this result as follows:

\begin{thm}[\cite{Nikiforov2016}, Proposition 4.49]\label{thmNikiforov2016}
If \( G \) is a connected graph of order \( n \), then \( \mathcal{E}_{2k}(G) \geq \mathcal{E}_{2k}(P_n) \) for every integer \( k \geq 2 \).
\end{thm} Thus, Theorem \ref{thmNikiforov2016} proves Conjecture \ref{conj1} (a) for every positive even integer \( p \). Recently, Arizmendi and Arizmendi \cite{Arizmendi2023} resolved Conjecture \ref{conj1} (b) in the case where \( G \) is a bipartite graph:

\begin{thm}[\cite{Arizmendi2023}, Proposition 4.7 (i)]\label{thmArizmendi2023}
Let $G$ be a bipartite graph of order $n$ and size $m$. If \(1 \leq p \leq 2\), then \( \mathcal{E}_p(G) \geq 2(m)^{p/2}\).
\end{thm}

In this paper, we establish the following theorem, which constitutes the main result of this paper.

\begin{thm}\label{mainthm}
Let \( G \) be a connected graph of order \( n \), and let \( S_n \) be the star graph. If \( 0 < p < 2 \), then
\[
\mathcal{E}_p(G) \geq \mathcal{E}_p(S_n),
\]
with equality holds if and only if \( G \) is the star \( S_n \).
\end{thm}

Hence, Theorem \ref{mainthm} confirms Conjecture \ref{conj1} (b) and additionally provides the equality condition, which was not mentioned in the original conjecture.

The paper is organized as follows. In Section~\ref{subsec21}, we derive a Coulson-Jacobs-type formula for the $p$-energy of connected graphs, which is more general than the original Coulson integral formula in \cite{Arizmendi2023b}, which applies only to bipartite graphs. In Section~\ref{subsec22}, we obtain sharper upper bounds for the $p$-energy when \( p > 2 \), which is crucial for deriving the key claim in the following section. In Section~\ref{sec3}, we provide the proof of Theorem~\ref{mainthm}. Some conclusions are given in Section~\ref{sec4}.

\section{Preliminaries}\label{sec2}
\subsection{A Coulson-Jacobs-type Formula}\label{subsec21}

In the theory of graph energy, the so-called \textit{Coulson integral formula} (\ref{eq1}) plays an outstanding role. This formula was obtained by Charles Coulson as early as 1940 \cite{Coulson1940} and reads:
\begin{equation}
\mathcal{E}(G) = \frac{1}{\pi} \int_{-\infty}^{+\infty} \left[ n - \frac{ix \, \phi'(G, ix)}{\phi(G, ix)} \right] dx = \frac{1}{\pi} \int_{-\infty}^{+\infty} \left[ n - x \frac{d}{dx} \log \phi(G, ix) \right] dx
\tag{1}\label{eq1}
\end{equation}
where \( G \) is a graph, \( \phi(G, x) \) is the characteristic polynomial of the adjacency matrix of graph \( G \), \( \phi'(G, x) = \frac{d}{dx}\phi(G, x) \) is its first derivative, and \( i = \sqrt{-1} \).

Later, some researchers extended these notations of graph energy to general complex polynomials and defined the general energy of polynomials \cite{Qiao2016}. Let
\[
\phi(z) = \sum_{k=0}^{n} a_k z^{n-k} = a_0 (z - z_1)(z - z_2) \cdots (z - z_n)
\]
be a complex polynomial (all the coefficients are complex numbers) of degree \(n\). For a nonzero real number \(\alpha\), the \textit{general energy} of \(\phi(z)\), denoted by \(E_\alpha(\phi)\), is defined as:
\[
E_\alpha(\phi) = \sum_{z_k \neq 0} |z_k|^\alpha.
\] In particular, when \( \phi(z) \) is the characteristic polynomial of the graph \( G \), we have \( E_\alpha(\phi) = \mathcal{E}_\alpha(G) \).

In the literature \cite{Du2021,Qiao2016,Qiao2017}, extensive research has focused on Coulson-type integral formulas for such polynomials. We now need to use the following formula from \cite{Du2021}:

\begin{thm}[\cite{Du2021}, Theorem 4.1 (i)]\label{Du1}
Let \( \phi(z) \) be a complex polynomial of degree \( n \), whose roots are all real numbers, \( c < n \) the multiplicity of 0 as a root of \( \phi(z) \). If \( 0 < \alpha < 2 \), then
\[
E_{\alpha}(\phi) = \frac{\sin \frac{\alpha \pi}{2}}{\alpha \pi} \int_0^{+\infty} \left( 2n + ix^\frac{1}{\alpha} \left(\frac{\phi'(-ix^\frac{1}{\alpha})}{\phi(-ix^\frac{1}{\alpha})} - \frac{\phi'(ix^\frac{1}{\alpha})}{\phi(ix^\frac{1}{\alpha})} \right) \right)dx.
\]
\end{thm} Our aim is to use the integral formula of Theorem \ref{Du1}, to compare the \( p \)-energy of two connected graphs, \( G_1 \) and \( G_2 \). For this purpose, we will need to modify it as done in the classical paper of Coulson and Jacobs \cite{CoulsonJacobs1949}.

\begin{prop}[Coulson-Jacobs-type
formula for the $p$-energy]\label{new1}
Let \( G_1 \) and \( G_2 \) be connected graphs on \( n \) vertices, and \( 0 < p < 2 \). Let \( f(z) \) and \( g(z) \) denote the characteristic polynomials of the adjacency matrices of \( G_1 \) and \( G_2 \), respectively. Then
\[
\mathcal{E}_p(G_1) - \mathcal{E}_p(G_2) = \frac{2p \sin \left( \frac{p\pi}{2} \right)}{\pi} \int_0^{+\infty} z^{p-1} \log \left| \frac{f(iz)}{g(iz)} \right| dz.
\]
\end{prop}

\begin{proof}
By performing the substitution \( z = x^{\frac{1}{p}} \), we have \( dx = p z^{p-1} dz \). Applying this to the integral in Theorem \ref{Du1}, we obtain:
\[
\mathcal{E}_p(G_1) = \frac{\sin\left(\frac{p \pi}{2}\right)}{\pi} \int_0^{+\infty} z^{p-1} \left[ 2n + iz \left( \frac{f'(-iz)}{f(-iz)} - \frac{f'(iz)}{f(iz)} \right) \right] dz.
\] Thus, \begin{align*}
\mathcal{E}_p(G_1) - \mathcal{E}_p(G_2) &= \frac{\sin\left(\frac{p \pi}{2}\right)}{\pi}  \int_0^{+\infty} iz^{p} \left[\left(\frac{f'(-iz)}{f(-iz)} - \frac{g'(-iz)}{g(-iz)}\right) - \left(\frac{f'(iz)}{f(iz)} - \frac{g'(iz)}{g(iz)}\right)\right] dz \\&= \frac{\sin\left(\frac{p \pi}{2}\right)}{\pi}  \int_0^{+\infty} z^p \frac{d}{dz} \left[ \log \left( \frac{g(-iz)  g(iz)}{f(-iz)  f(iz)} \right) \right] dz.
\end{align*} We proceed by applying integration by parts: \begin{align*}
\int_0^{+\infty} z^p \frac{d}{dz} \log \left( \frac{g(-iz)  g(iz)}{f(-iz)  f(iz)} \right) dz =& - \int_0^{+\infty} p z^{p-1} \log \left( \frac{g(-iz)  g(iz)}{f(-iz)  f(iz)} \right) dz
\\&+ \left[ z^p \log \left( \frac{g(-iz)  g(iz)}{f(-iz)  f(iz)} \right) \right]_0^\infty.\tag{2}\label{eq2}
\end{align*}
Now, we calculate the last term in (\ref{eq2}). On the one hand,
\[
\lim_{z \to 0^+} z^p \log \left( \frac{g(-iz)g(iz)}{f(-iz)f(iz)} \right) = \lim_{z \to 0^+} \left[ z^p  \log \left( g(-iz)g(iz) \right) - z^p  \log \left( f(-iz)f(iz) \right) \right] = 0
\] since for any \( p > 0 \) and any non-trivial polynomial \( P(z) \), we have the convergence \( z^p \log(P(z)) \to 0 \) as \( z \to 0 \) (which may be proved by an application of L'Hospital's rule).

On the other hand, since \( f \) and \( g \) are the characteristic polynomials of the graphs \( G_1 \) and \( G_2 \), respectively, the sum of their roots is zero because the trace of the adjacency matrix is zero. Hence, we can express them as
\[
f(z) = z^n + a_1 z^{n-2} + o(z^{n-2}), \quad g(z) = z^n + b_1 z^{n-2} + o(z^{n-2}),
\]
as \( z \to \infty \). This implies that
\[
\frac{g(iz)}{f(iz)} = 1 + (a_1 - b_1) z^{-2} + o(z^{-2}).
\] Thus,
\[
\frac{g(-iz)g(iz)}{f(-iz)f(iz)} = 1 + 2(a_1 - b_1) z^{-2} + o(z^{-2}).
\] Therefore, using the approximation $\log(1 + z) = z$ for \(z \to 0\), we see that
\begin{eqnarray*}
z^p  \log \left( \frac{g(-iz)g(iz)}{f(-iz)f(iz)} \right) &=& z^p \log \left( 1 + 2(a_1 - b_1) z^{-2} + o(z^{-2}) \right)
\\&=&  2(a_1 - b_1) z^{p-2} + o(z^{p-2}) ,
\end{eqnarray*} which converges to $0$  as  $z \to \infty$,  for any  $0< p < 2$. Hence,
\begin{equation}
\mathcal{E}_p(G_1) - \mathcal{E}_p(G_2) = \frac{p \sin\left( \frac{p \pi}{2} \right)}{\pi} \int_0^{+\infty} z^{p-1} \log \left( \frac{f(-iz)f(iz)}{g(-iz)g(iz)} \right) dz.
\tag{3}\label{eq3}
\end{equation}
The integrand in Eq.~(\ref{eq3}) may be complex valued, although its left-hand side is necessarily real. In view of the fact that the real part of $\log z$ is $\log |z|$, we can rewrite Eq.~(\ref{eq3}) as
\begin{equation}
\mathcal{E}_p(G_1) - \mathcal{E}_p(G_2) = \frac{p \sin \left( \frac{p\pi}{2} \right)}{\pi} \int_0^{+\infty} z^{p-1} \log \left| \frac{f(-iz)f(iz)}{g(-iz)g(iz)} \right| dz.
\tag{4}\label{eq4}
\end{equation}
Moreover, since the coefficients of $f(z)$ and $g(z)$ are real, we have \(f(iz)f(-iz)=|f(iz)|^2\) for \(z \in \mathbb{R}\). Hence,
\begin{equation}
\mathcal{E}_p(G_1) - \mathcal{E}_p(G_2) = \frac{2p \sin \left( \frac{p\pi}{2} \right)}{\pi} \int_0^{+\infty} z^{p-1} \log \left| \frac{f(iz)}{g(iz)} \right| dz,
\tag{5}\label{eq5}
\end{equation}
proving Proposition \ref{new1}. \end{proof}

\begin{remark}
When \( p = 1 \), Proposition \ref{new1} is equivalent to the famous \textit{Coulson–Jacobs formula} \cite{CoulsonJacobs1949}:
\begin{equation}
\mathcal{E}(G_1) - \mathcal{E}(G_2) = \frac{1}{\pi} \int_{-\infty}^{+\infty} \log \left| \frac{\phi(G_1, ix)}{\phi(G_2, ix)} \right| dx,
\tag{6}\label{eq6}
\end{equation}
where \( \phi(G_1, x) \) and \( \phi(G_2, x) \) are the characteristic polynomials of the adjacency matrices of graphs \( G_1 \) and \( G_2 \), respectively. This is the reason why our Proposition \ref{new1} is named the Coulson-Jacobs-type formula for the $p$-energy.
\end{remark}

\subsection{Upper Bounds for \(p\)-energy when \(p > 2\)}\label{subsec22}

We begin by recalling an upper bound on the spectral radius of connected graphs proved in Hong \cite{Hong1988}:

\begin{thm}[\cite{Hong1988}, Theorem 1]\label{lemmahong}
Let \( G \) be a connected simple graph with \( m \) edges and \( n \) vertices. Then the spectral radius \(\rho(A)\) of the adjacency matrix \( A \) of a graph \( G \) satisfies
\[
\rho(A) \leq \sqrt{2m - n + 1},
\]
with equality if and only if \( G \) is isomorphic to one of the following two graphs:
\begin{itemize}
    \item[(a)] the star \( S_n \),
    \item[(b)] the complete graph \( K_n \).
\end{itemize}
\end{thm}

By Theorem \ref{lemmahong}, we can derive the following upper bounds on \(\mathcal{E}_p(G)\) when \(p > 2\).

\begin{prop}\label{new2}
Let \( p > 2 \) and \( G \) be a connected graph with \(m\) edges and \( n \) vertices. Then
\begin{equation}
  \mathcal{E}_p(G) = \| G \|_{p}^p \leq 2m \left( 2m - n + 1 \right)^{\frac{p-2}{2}},\tag{7}\label{eq7}
\end{equation}
with equality if and only if \( G \) is the star \( S_n \).
\end{prop}

\begin{proof}
Let $\sigma_1(G) \geq \sigma_2(G) \geq \dots \geq \sigma_n(G)$ be the singular values of \(A(G)\). Since \( G \) is a graph of size \( m \), we have
\begin{equation}
\sigma_1^2(G) + \cdots + \sigma_n^2(G) = 2m.\tag{8}\label{eq8}
\end{equation}
We want to maximize \( \mathcal{E}_p(G) \) subject to (\ref{eq8}). First, note that
\[
\sigma_1^p(G) + \cdots + \sigma_n^p(G) \leq \sigma_1^{p-2}(G) (\sigma_1^2(G) + \cdots + \sigma_n^2(G)) = 2m\cdot\sigma_1^{p-2}(G).
\]
On the other hand, Theorem \ref{lemmahong} implies that
\[
\sigma_1(G) = \rho(G) \leq \sqrt{2m - n + 1},
\]
and (\ref{eq7}) follows.

If equality holds in (\ref{eq7}), then \(\sigma_1(G) = \sqrt{2m - n + 1}\) and \(\sigma_1^{p-2}(G)\sigma_k^2(G)=\sigma_k^p(G)\) holds for \(1\leq k \leq n\). Thus, by Theorem \ref{lemmahong}, we conclude that \(G\) can be either the star \(S_n\) or the complete graph $K_n$. However, since all singular values satisfy $\sigma_i(G)\ge 0$, it follows that either $\sigma_k(G) = \sigma_1(G)$, or $\sigma_k(G) = 0$. For a complete graph \( K_n \), its singular values are \( n-1 \) with multiplicity \( 1 \) and \( 1 \) with multiplicity \( n-1 \), which does not satisfy the condition above when \( n \geq 3 \). Therefore, the only possible graph \( G \) is the star graph \( S_n \), which proves Proposition \ref{new2}.
\end{proof}

\begin{remark}
If we remove the condition that \( G \) is connected, for a general simple graph, Nikiforov \cite[Proposition 4.26]{Nikiforov2016} provides the following upper bounds on \(\mathcal{E}_p(G)\) for \(p>2\): \begin{equation}
\mathcal{E}_p(G) = \|G\|_p^p < 2m \left( -\frac{1}{2} + \sqrt{2m + \frac{1}{4}} \right)^{p-2}.
\tag{9}\label{eq9}
\end{equation}
When \( G \) is a connected graph, our Theorem \ref{new2} is sharper than \eqref{eq9}. To verify this, we only need to check that \(
2m \left( 2m - n + 1 \right)^{\frac{p-2}{2}} \leq 2m \left( -\frac{1}{2} + \sqrt{2m + \frac{1}{4}} \right)^{p-2},
\) which is equivalent to \( m \leq \frac{n(n-1)}{2} \).
\end{remark}

As a direct corollary of Proposition \ref{new2}, we establish an upper bound for \( \mathcal{E}_4(G) \), which will play a crucial role in deriving the key claim in the next section.

\begin{cor}\label{corr}
Let \( G \) be a connected graph with \(m\) edges and \( n \) vertices. Then
\[
  \mathcal{E}_4(G) \leq 2m \left( 2m - n + 1 \right),
\]
with equality if and only if \( G \) is the star \( S_n \).
\end{cor}

\section{Proof of Theorem \ref{mainthm}}\label{sec3}

We are now ready to present the following:
\begin{proof}[Proof of Theorem \ref{mainthm}] Let \(0<p<2\). It suffices to prove that for any connected graph \( G \), we have \( \mathcal{E}_p(G) - \mathcal{E}_p(S_n) \geq 0 \). Given that the characteristic polynomial of \( S_n \) is \( g(z) = z^n - (n-1)z^{n-2} \), Proposition \ref{new1} provides the expression:
\[
\mathcal{E}_p(G) - \mathcal{E}_p(S_n) =\frac{2p \sin \left( \frac{p\pi}{2} \right)}{\pi}\int_0^{+\infty} z^{p-1}  \log \left| \frac{f(iz)}{z^{n-2}(z^2+n-1)} \right| dz.
\] Let \[I = \int_0^{+\infty} z^{p-1}  \log \left| \frac{f(iz)}{z^{n-2}(z^2+n-1)} \right| dz. \]
What remains is to prove that \(I \geq 0\).

Now we claim that \(|f(ix)| \geq |x^{n-2}(x^2+n-1)|\) for \(x \geq 0\).

Let \( \lambda_1, \dots, \lambda_n \) be the eigenvalues of \( A(G) \). Then \( f(x) = \prod_{k=1}^n (x - \lambda_k) \), and \( f(ix) = \prod_{k=1}^n (ix - \lambda_k) \). Since \( \lambda_k \in \mathbb{R} \), we have \begin{align*}
|f(ix)|^2 &= \prod_{k=1}^n |ix - \lambda_k|^2 = \prod_{k=1}^n (x^2 + \lambda_k^2) \\&\geq x^{2n} + \left( \sum_{k=1}^n \lambda_k^2 \right) x^{2n-2} + \left( \sum_{1 \leq i < j \leq n} \lambda_i^2 \lambda_j^2 \right) x^{2n-4}.
\tag{10}\label{eq10}
\end{align*}
It is known that \( \sum_{k=1}^n \lambda_k^2 = 2m \) and \( m \geq n-1 \) for connected graphs, thus
\begin{align*}
|f(ix)|^2 \geq x^{2n} + 2(n-1) x^{2n-2} + \left( \sum_{1 \leq i < j \leq n} \lambda_i^2 \lambda_j^2 \right) x^{2n-4}.
\tag{11}\label{eq11}
\end{align*}
Moreover, since \( \left( \sum_{k=1}^n \lambda_k^2 \right)^2 = \sum_{k=1}^n \lambda_k^4 + 2 \sum_{1 \leq i < j \leq n} \lambda_i^2 \lambda_j^2 \), we have
\begin{align*}
\sum_{1 \leq i < j \leq n} \lambda_i^2 \lambda_j^2 = \frac{1}{2} \left( 4m^2 - \sum_{k=1}^n \lambda_k^4 \right) = 2m^2 - \frac{1}{2}\mathcal{E}_4(G).
\tag{12}\label{eq12}
\end{align*}
By Corollary \ref{corr}, we know \( \mathcal{E}_4(G) \leq 2m (2m - n + 1) \), hence
\begin{align*}
\sum_{1 \leq i < j \leq n} \lambda_i^2 \lambda_j^2 &\geq 2m^2 - m(2m - n + 1) \\&= m(n-1) \geq (n-1)^2.
\tag{13}\label{eq13}
\end{align*}
Therefore, substituting equation (\ref{eq13}) into equation (\ref{eq11}), we obtain
\begin{align*}
|f(ix)|^2 &\geq x^{2n} + 2(n-1)x^{2n-2} + (n-1)^2x^{2n-4} \\&= |x^{n-2}(x^2+n-1)|^2,
\end{align*} which proves our claim that \(|f(ix)| \geq |x^{n-2}(x^2+n-1)|\).

Now we know that \(\log \left| \frac{f(iz)}{z^{n-2}(z^2+n-1)} \right| \geq 0\) for \(z \geq 0\), hence the integral
\begin{align*}
I = \int_0^{+\infty} z^{p-1}  \log \left| \frac{f(iz)}{z^{n-2}(z^2+n-1)} \right| dz \geq 0 \tag{14}\label{eq14}
\end{align*}
holds for any connected graph \(G\). Therefore, we have proved the inequality in Theorem \ref{mainthm}, namely, \begin{align*} \mathcal{E}_p(G) \geq \mathcal{E}_p(S_n). \tag{15}\label{eq15}
\end{align*}

If the equality holds in (\ref{eq15}), then \(\log \left| \frac{f(iz)}{z^{n-2}(z^2+n-1)} \right| = 0\) for \(z \geq 0\), which means \(|f(ix)| = |x^{n-2}(x^2+n-1)|\) for \(x \geq 0\), hence \( \mathcal{E}_4(G) = 2m (2m - n + 1) \). From the equality condition in Corollary \ref{corr}, we conclude that \( G \) can only be the star graph \( S_n \). This completes the proof of Theorem \ref{mainthm}. \end{proof}

\section{Concluding Remarks}\label{sec4}

In this paper, we resolved a conjecture proposed by Nikiforov (Conjecture \ref{conj1} (b)) on the minimal \( p \)-energy of connected graphs, by using the Coulson-Jacobs-type formula. We established that for \( 0 < p < 2 \), the minimal \( p \)-energy among all connected graphs is attained only by the star graph \( S_n \).

A natural question arises: can the same approach be used to address Conjecture \ref{conj1} (a) for \( p > 2 \)? To explore this, let us first examine the form of the Coulson-type integral formula for the general energy when \( \alpha \geq 2  \).
\begin{thm}[\cite{Du2021}, Theorem 4.1 (iii)]\label{Du2}
Let \( \phi(z) \) be a complex polynomial of degree \( n \), whose roots are all real numbers, \( c < n \) the multiplicity of 0 as a root of \( \phi(z) \). If \( \alpha \geq 2 \), then for any positive even integer \( r > \alpha \),
\[
    E_{\alpha}(\phi) = \frac{r \sin \frac{\alpha \pi}{r}}{2 \alpha \pi}
    \int_0^{+\infty}
    \left( 2n + \mathrm{i}x^{\frac{r}{2\alpha}}
    \left( \frac{\psi'(-\mathrm{i}x^{\frac{r}{2\alpha}})}{\psi(-\mathrm{i}x^{\frac{r}{2\alpha}})}
    - \frac{\psi'(\mathrm{i}x^{\frac{r}{2\alpha}})}{\psi(\mathrm{i}x^{\frac{r}{2\alpha}})} \right)
    \right) dx,
\] where \[\psi(z) = e^{\mathrm{i} (\frac{r}{2} - 1) n\pi}
    \phi ( z^{\frac{2}{r}} )
    \phi ( z^{\frac{2}{r}} e^{-\mathrm{i} \frac{4\pi}{r}} )
    \cdots
    \phi ( z^{\frac{2}{r}} e^{-\mathrm{i} \frac{4(r/2-1) \pi}{r}} ).\]
\end{thm}

Using the same proof technique as in Proposition \ref{new1}, we can easily derive the difference in \( p \)-energy between two connected graphs when \( p > 2 \):

\begin{prop}
Let \( G_1 \) and \( G_2 \) be connected graphs on \( n \) vertices, and let \( p > 2 \). Denote by \( f(z) \) and \( g(z) \) the characteristic polynomials of the adjacency matrices of \( G_1 \) and \( G_2 \), respectively. Let \(r\) be a positive even integer such that \( p < r < 2p \). Then
\[
\mathcal{E}_p(G_1) - \mathcal{E}_p(G_2) = \frac{4p \sin \left( \frac{p\pi}{r} \right)}{r\pi} \int_0^{+\infty} z^{\frac{2p}{r}-1} \log \left| \frac{\tilde{f}(iz)}{\tilde{g}(iz)} \right| dz,
\] where \[\tilde{f}(z) = e^{\mathrm{i} (\frac{r}{2} - 1) n\pi}
    f ( z^{\frac{2}{r}} )
    f ( z^{\frac{2}{r}} e^{-\mathrm{i} \frac{4\pi}{r}} )
    \cdots
    f ( z^{\frac{2}{r}} e^{-\mathrm{i} \frac{4(r/2-1) \pi}{r}} )\] and \[\tilde{g}(z) = e^{\mathrm{i} (\frac{r}{2} - 1) n\pi}
    g ( z^{\frac{2}{r}} )
    g ( z^{\frac{2}{r}} e^{-\mathrm{i} \frac{4\pi}{r}} )
    \cdots
    g ( z^{\frac{2}{r}} e^{-\mathrm{i} \frac{4(r/2-1) \pi}{r}} ).\]
\end{prop}

It is known (see, e.g. \cite[2.6.7]{Cvetkovic1980}) that the characteristic polynomial of \( P_n \) is given by
\[
g(\lambda) = U_n\left(\frac{\lambda}{2} \right) = \sum_{k=0}^{\lfloor n/2 \rfloor} (-1)^k \binom{n-k}{k} \lambda^{n-2k},
\] where \(U_n(x)\) is the Chebyshev polynomial of the second kind,
\[
U_n(x) = \frac{\sin \left[ (n+1) \arccos x \right]}{\sqrt{1-x^2}}.
\]

Thus, to prove Conjecture \ref{conj1} (a), it suffices to show that for any connected graph \( G \), we have
\[
\mathcal{E}_p(G) - \mathcal{E}_p(P_n) \geq 0,
\]
which is equivalent to proving that
\[
I = \int_0^{+\infty} x^{\frac{2p}{r}-1} \log \left| \frac{\tilde{f}(ix)}{\tilde{g}(ix)} \right| dx \geq 0.
\]

One might attempt to prove that \( |\tilde{f}(iz)| \geq |\tilde{g}(ix)| \) holds for every \( x \geq 0 \), which is equivalent to showing that
\begin{align*}
&\left| f \left( x^{\frac{2}{r}} e^{\mathrm{i} \frac{\pi}{r}} \right)
    f \left( x^{\frac{2}{r}} e^{-\mathrm{i} \frac{3\pi}{r}} \right)
    \cdots
    f \left( x^{\frac{2}{r}} e^{-\mathrm{i} \frac{(2r-5) \pi}{r}} \right) \right|
\\&\geq
\left| g \left( x^{\frac{2}{r}} e^{\mathrm{i} \frac{\pi}{r}} \right)
    g \left( x^{\frac{2}{r}} e^{-\mathrm{i} \frac{3\pi}{r}} \right)
    \cdots
    g \left( x^{\frac{2}{r}} e^{-\mathrm{i} \frac{(2r-5) \pi}{r}} \right) \right|,
\tag{16}\label{eq16}
\end{align*} where \( f \) is the characteristic polynomial of an arbitrary connected graph \( G \).

Unfortunately, this approach proves to be quite challenging. For instance, when the adjacency matrix of \( G \) is singular and \( n \) is even, we have \( |g(0)| = 1 \) but \( |f(0)| = 0 \). As a result, for \( x \) sufficiently close to \( 0 \), inequality (\ref{eq16}) cannot hold. Moreover, when \( n \) is odd, \( 0 \) is a simple zero of \( g(x) \) and \( f(x) \) has a zero at \( 0 \) with order greater than one, then similarly, inequality (\ref{eq16}) also cannot hold for \( x \) sufficiently close to \( 0 \).

Furthermore, it remains unknown whether the equality in Conjecture \ref{conj1} (a) holds exclusively for \( G = P_n \). In addition, the spectral radius bound for connected graphs established by Hong \cite{Hong1988} is unlikely to be applicable, as its equality condition is attained by \( S_n \) or \( K_n \), rather than \( P_n \). This suggests that establishing some spectral inequalities where the equality condition holds specifically for \( P_n \) may be a necessary step toward fully resolving this conjecture.

\section*{Acknowledgments}
This work is supported by National Key Research and Development Program of China 2023YFA1010200, National Natural Science Foundation of China
	under grant No.\,12371357.

The authors would like to express their sincere gratitude to Dr.~Clive Elphick for his valuable comments, which have significantly improved an earlier version of this paper.

\end{document}